\documentclass[11pt]{article}
\usepackage{graphicx}
\usepackage{amsthm, amsmath, amssymb, tikz, bm}
\usepackage{mathtools}
\usepackage{xifthen}
\usepackage{caption}
\usepackage[colorlinks=true,
linkcolor=blue,citecolor=blue,
urlcolor=blue]{hyperref}

\usepackage[margin=0.8in]{geometry} 

\usepackage[shortlabels]{enumitem}
\usepackage{todonotes}
\usetikzlibrary{calc,shapes, backgrounds}
\allowdisplaybreaks

\newtheorem{lemma}{Lemma}
\newtheorem{theorem}{Theorem}
\newtheorem{corollary}{Corollary}

\newtheorem{claim}{Claim}

\newcommand{\dss}{\displaystyle\sum}

\newcommand{\lp}{\left(}
\newcommand{\rp}{\right)}

\newcommand{\cx}{{\bf x}}
\newcommand{\cz}{{\bf z}}
\newcommand{\vol}{{\rm Vol}}

\DeclarePairedDelimiter{\abs}{\lvert}{\rvert}%

\newcommand{\LL}[1]{\textcolor{blue}{{}Comment by LL: #1}}

\newcommand{\WL}[1]{\textcolor{magenta}{{}Comment by WL: #1}}

\newcommand{\one}{\mathbf{1}}

\title{Maximum spread of $K_{s,t}$-minor-free graphs} 
\author{
William Linz, \thanks{University of South Carolina, Columbia, SC. ({\tt wlinz@mailbox.sc.edu}). The author is partially supported by NSF DMS 2038080 grant.}
\and Linyuan Lu, \thanks{University of South Carolina, Columbia, SC. ({\tt lu@math.sc.edu}). The author is partially supported by NSF DMS 2038080 grant.}
\and Zhiyu Wang \thanks{Louisiana State University, Baton Rouge, LA.
({\tt zhiyuw@lsu.edu}). This author was supported in part by LA Board of Regents grant LEQSF(2024-27)-RD-A-16.}
}

\begin{document}

\maketitle
\abstract{
The spread of a graph $G$ is the difference between the largest and smallest eigenvalue of the adjacency matrix of $G$. 
In this paper, we consider the family of graphs which contain no $K_{s,t}$-minor.
We show that for any $t\geq s \geq  2$ and sufficiently large $n$, there is an integer $\xi_{t}$ such that the extremal $n$-vertex $K_{s,t}$-minor-free graph attaining the maximum spread is the graph obtained by joining a graph $L$ on $(s-1)$ vertices to the disjoint union of $\lfloor \frac{2n+\xi_{t}}{3t}\rfloor$ copies of $K_t$ and
$n-s+1 - t\lfloor \frac{2n+\xi_t}{3t}\rfloor$
isolated vertices. Furthermore, we give an explicit formula for $\xi_{t}$ and an explicit description for the graph $L$ for $t \geq \frac32(s-3) +\frac{4}{s-1}$. 
}

\section{Introduction}
Given a square matrix $M$, the \textit{spread} of $M$, denoted by $S(M)$, is defined as $S(M):= \max_{i,j} |\lambda_i -\lambda_j|$, where the maximum is taken over all pairs of eigenvalues of $M$, so that $S(M)$ is the diameter of the spectrum of $M$.
Given a graph $G=(V,E)$ on $n$ vertices, the \textit{spread} of $G$, denoted by $S(G)$, is defined as the spread of the adjacency matrix $A(G)$ of $G$. The adjacency matrix $A(G)$ is the $n\times{n}$ matrix with rows and columns indexed by the vertices of $G$ such that for every pair of vertices $u, v\in V(G)$, $(A(G))_{uv} = 1$ if $uv\in E(G)$ and $(A(G))_{uv} = 0$ otherwise. Since $A(G)$ is a real symmetric matrix, its eigenvalues are all real numbers. Let $\lambda_1(G) \geq \cdots \geq \lambda_n(G)$ be the eigenvalues of $A(G)$, where $\lambda_1$ is called the $\textit{spectral radius}$ of $G$. Then $S(G) = \lambda_1 -\lambda_n$. 

The systematic study of the spread of graphs was initiated by Gregory, Hershkowitz, and Kirkland \cite{GHK2001}. One of the central focuses of this area is to find the maximum or minimum spread over a fixed family of graphs and characterize the extremal graphs. The maximum-spread graph over the family of all $n$-vertex graphs was recently determined for sufficiently large $n$ by Breen, Riasanovsky, Tait and Urschel~\cite{BRTU2021+}, building on much prior work~\cite{Aouchiche2008, Riasanovsky2021,Stanic2015, Stevanovic2014, Urschel2021}. Other problems of such extremal flavor have been investigated for trees~\cite{AP2015}, graphs with few cycles~\cite{FWG2008, PBA2009, Wu-Shu2010}, the family of bipartite graphs~\cite{BRTU2021+}, graphs with a given matching number~\cite{LZZ2007}, girth~\cite{WZS2013}, or size~\cite{Liu-Liu2009}, outerplanar graphs~\cite{GBT2022, LLLW2022} and planar graphs~\cite{LLLW2022}. We note that the spreads of other matrices associated with a graph have also been extensively studied (see e.g. references in \cite{GBT2022, Cao-Vince93, Cvetkovic-Rowlinson90}).

Given two graphs $G$ and $H$, the \textit{join} of $G$ and $H$, denoted by $G\vee H$, is the graph obtained from the disjoint union of $G$ and $H$ by connecting every vertex of $G$ with every vertex of $H$. Let $P_k$ denote the path on $k$ vertices. Given two graphs $G$ and $H$, let $G\cup H$ denote the disjoint union of $G$ and $H$. Given a graph $G$ and a positive integer $k$, we use $kG$ to denote the disjoint union of $k$ copies of $G$. Given $v \subseteq V(G)$, let $N_G(v)$ denote the set of neighbors of $v$ in $G$, and let $d_G(v)$ denote the degree of $v$ in $G$, i.e., $d_G(v) = |N_G(v)|$. Given $S\subseteq V(G)$, define $N_G(S)$ as $N_G(S) =\cup_{v\in S} (N_G(v)\backslash S)$. 
We may ignore the subscript $G$ when there is no ambiguity. A graph $H$ is called a \textit{minor} of a graph $G$ if a graph isomorphic to
$H$ can be obtained from a subgraph of G by contracting edges. A graph $G$ is called \textit{$H$-minor-free} if $H$ is not a minor of $G$. 

There has been extensive work on finding the maximum spectral radius of $K_{s, t}$-minor-free graphs. 
Let ${\cal G}_{s,t}(n)$ denote the family of all $K_{s, t}$-minor-free graphs on $n$ vertices.
Nikiforov~\cite{Nikiforov2017} proved an upper bound for the maximum spectral radius of a $K_{2, t}$-minor-free graph. Nikiforov showed that this bound is tight for graphs with a sufficiently large number of vertices $n$ with $n\equiv 1\pmod{t}$ and determined the extremal graph in these cases. Tait~\cite{Tait2019} extended Nikiforov's result by proving an upper bound on the maximum spectral radius of $K_{s, t}$-minor-free graphs, and determined the extremal graphs when $n\equiv s-1\pmod{t}$ and $n$ is sufficiently large. Recently, Zhai and Lin~\cite{ZL2022} completely determined the $K_{s, t}$-minor-free graph with maximum spectral radius for a sufficiently large number of vertices $n$ and all $t\ge s\ge 2$.

In \cite{LLW2023}, the authors determined the maximum-spread $K_{2, t}$-minor-free graph for sufficiently large $n$ for all $t\ge 2$. In this follow-up paper, we determine the structure of the maximum-spread $K_{s, t}$-minor-free graph on $n$ vertices for sufficiently large $n$ and for all $t\geq s\geq 2$.

\begin{theorem} \label{thm:main}
    For $t\geq s \geq 2$ and $n$ sufficiently large, the graph(s) that maximizes the spread over the family of $K_{s,t}$-minor-free graphs on $n$ vertices has the following form 
    $$L_{max}\vee \left( \ell_0 K_t \cup \left(n-s+1- t \ell_0\right)  P_1\right)$$  where
    \begin{enumerate}
        \item $L_{max}$ is a graph on $s-1$ vertices which maximizes a function $\psi(L)$ (over all graphs $L$ on $s-1$ vertices) as follows:
        \begin{equation} \label{eq:psi}
            \psi(L)=3\sum_{v\in V(L)}d_L^2(v)-\frac{2}{s-1} \left(\sum_{v\in V(L)} d_L(v)\right)^2 - (t-1)\sum_{v\in V(L)} d_L(v).
        \end{equation}

        \item $\ell_0=\left(\frac{2}{3t}-\frac{2|E(L_{max})|}{3t(t-1)(s-1)}\right)(n-s+1) +O(n^{\epsilon})$ for any $\epsilon>0$.
    \end{enumerate}
    In particular, we have
    \begin{equation}
        \max_{G\in {\cal G}_{s,t}(n)}S(G)=  2\sqrt{(s-1)(n-s+1)} + \frac{(t-1)^2+\psi(L_{max})/(s-1)}{3\sqrt{(s-1)(n-s+1)}} + O\lp\frac{1}{n^{3/2}}\rp.
    \end{equation}
   
\end{theorem} 

We call a pair $(s,t)$ \textit{admissible} if $L_{max}=(s-1)K_1$, i.e.,
$\psi(L)\leq 0$ and $\psi(L)=0$ only if $L=(s-1)K_1$. We determine the value of $\ell_0$ when $(s,t)$ is admissible and thus determine the precise extremal graph(s) for these cases.

\begin{theorem}\label{thm:admissible}
Let $s$ and $t$ be integers with $t\geq s\geq 2$, and suppose that the pair $(s, t)$ is admissible. 
For $n$ sufficiently large, the maximum spread over the family of $K_{s,t}$-minor-free graphs on $n$ vertices is achieved by 
$$(s-1)K_1\vee  \left( \ell_0 K_t \cup \left(n-s+1- t\ell_0\right)  P_1\right).$$
Here $\ell_0$ is the nearest integer(s) of $\ell_1:=\frac{2}{3t}\left(n-s+1-\frac{(t-1)^2}{9(s-1)}\right)$. In particular, the extremal graph is unique when $\ell_1$ is not a half-integer. Otherwise, there are two extremal graphs.
\end{theorem}

Furthermore, we determine all of the admissible pairs $(s, t)$.

\begin{theorem}\label{thm:admpairs}
A pair $(s, t)$ with $s\le t$ is admissible if and only if $t \geq \frac32(s-3) +\frac{4}{s-1}$.
\end{theorem}

The smallest non-admissible pair is $(s,t)=(8,8)$.

Our paper is organized as follows. In Section \ref{sec:notation_lemmas}, we recall some useful lemmas and prove that in any maximum-spread $K_{s, t}$-minor-free graph $G$, there are $(s-1)$ vertices $u_1,\ldots, u_{s-1}$ which are adjacent to all other vertices in $G$. In Section \ref{sec:main_thm}, we show that $G - \{u_1,\ldots, u_{s-1}\}$ is a disjoint union of cliques on $t$ vertices and isolated vertices and complete the proofs of Theorems~\ref{thm:main}, \ref{thm:admissible}, and \ref{thm:admpairs}. The non-admissible cases are more complicated and will be handled in a sequel.

\section{Notation and Lemmas}\label{sec:notation_lemmas}
Let $G$ be a graph which attains the maximum spread among all $n$-vertex $K_{s,t}$-minor-free graphs, and $\lambda_1\geq \cdots \geq \lambda_n$ be the eigenvalues of $A(G)$. 
We first recall the following result by Mader \cite{Mader1967}.

\begin{theorem}[\cite{Mader1967}]\label{thm:linear-edge-bound}
    For every positive integer $t$, there exists a constant $C_t$ such that every graph with average degree at least $C_t$ contains a $K_t$ minor.
\end{theorem}

\begin{corollary}\label{cor:Mader}
Let $s$ and $t$ be positive integers with $s\leq t$. There exists a constant $C_0$ such that for any $K_{s,t}$-minor-free graph $G$ on $n>0$ vertices,
$$|E(G)|\leq C_0n.$$
\end{corollary}

Kostochka and Prince~\cite{KP08} gave a better upper bound on the maximum number of edges in a $K_{s,t}$-minor-free graph when $t$ is sufficiently large compared to $s$.
\begin{theorem}\cite{KP08}\label{thm:KP_extremal}
Let $t \ge (180s\log_2s)^{1+6s\log_2s}$ be a positive integer, and $G$ be a graph on $n\geq s+t$ vertices with no $K_{s,t}$ minor. Then
$$|E(G)|\leq \frac{t+3s}{2}(n-s+1).$$
\end{theorem}

We mention here that in the case of $s=2$, Chudnovsky, Reed and Seymour~\cite{CRS2011} showed a tight upper bound
$|E(G)|\leq \frac{1}{2}(t+1)(n-1)$ for the number of edges in a $K_{2, t}$-minor-free graph $G$ for any $t\ge 2$, which extends an earlier result of Myers \cite{Myers2008}.

We also need the following theorem by Thomason \cite{Thomason2007} on the number of edges of $K_{s,t}$-minor-free bipartite graphs.

\begin{theorem}\cite{Thomason2007}\label{thm:bipartite_edge_bound}
Let $G$ be a bipartite graph with at least $(s-1)n + 4^{s+1}s!tm$ edges, where $n,m>0$ are the sizes of the two parts of $G$. Then $G$ has a $K_{s,t}$-minor.
\end{theorem}

\begin{corollary}\label{cor:star-minor}
Suppose $G$ is a bipartite graph on $n$ vertices such that one part has at most $c \sqrt{n}$ vertices for some fixed constant $c>0$. If $G$ is $K_{1,t}$-minor-free, then $|E(G)| < 16tc\sqrt{n}$.
\end{corollary}

As a first step towards proving Theorem \ref{thm:main}, we want to show that $G$ must contain $K_{s-1, n-s+1}$ as a subgraph. We recall the result of Tait~\cite{Tait2019} on the maximum spectral radius of $K_{s,t}$-minor-free  graphs.

\begin{theorem}\cite{Tait2019}\label{thm:taitspecrad}
Let $t \ge s\ge 2$ and let $G$ be a graph of order $n$ with no $K_{s,t}$ minor. For sufficiently large $n$, the spectral radius $\lambda_1(G)$ satisfies
$$\lambda_1(G)\leq \frac{s+t-3+\sqrt{(t-s+1)^2+4(s-1)(n-s+1)}}{2},$$
with equality if and only if $n\equiv s-1\pmod{t}$ and $G=K_{s-1}\vee \lfloor n/t\rfloor K_t$.
\end{theorem}

We first give some upper and lower bounds on $\lambda_1(G)$ and $|\lambda_n(G)|$ when $n$ is sufficiently large. We use known expressions for the eigenvalues of a join of two regular graphs~\cite[pg.19]{BH2012}.

\begin{lemma}\cite{BH2012}\label{lem:joinreglemma}
Let $G$ and $H$ be regular graphs with degrees $k$ and $\ell$ respectively. Suppose that $|V(G)| = m$ and $|V(H)| = n$. Then, the characteristic polynomial of $G\vee H$ is $p_{G\vee H}(t) = ((t-k)(t-\ell)-mn)\frac{p_G(t)p_H(t)}{(t-k)(t-\ell)}$. In particular, if the eigenvalues of $G$ are $k = \lambda_1 \ge \ldots \ge \lambda_m$ and the eigenvalues of $H$ are $\ell = \mu_1 \ge \ldots \geq \mu_n$, then the eigenvalues of $G\vee H$ are $\{\lambda_i: 2\le i\le m\} \cup \{\mu_j: 2\le j\le n\} \cup \{x: (x-k)(x-\ell)-mn = 0\}$. 
\end{lemma}

We will apply Lemma~\ref{lem:joinreglemma} to the graph  $(s-1)K_1\vee qK_t$ where $q=\lfloor (n-s+1)/t\rfloor$. Let $a_0 = (s-1)(n-s+1)$.

\begin{lemma}\label{lem:lambdan}
We have
\begin{equation}\label{lem:lambda}
\sqrt{a_0} - \frac{s+t-3}{2}-O\left(\frac{1}{\sqrt{n}}\right) \le |\lambda_n| \le \lambda_1 \le \sqrt{a_0} +\frac{s+ t - 3}{2}+O\left(\frac{1}{\sqrt{n}}\right).
\end{equation}
\end{lemma}

\begin{proof}
The upper bound of $\lambda_1$ is due to Theorem \ref{thm:taitspecrad}. Now let us prove the lower bound.

By Lemma~\ref{lem:joinreglemma}, for sufficiently large $n$, $\lambda_1((s-1)K_1\vee qK_t)$ and $\lambda_n((s-1)K_1\vee qK_t)$ are the roots of the equation
$$\lambda(\lambda-(t-1))-(s-1)qt=0.$$
Thus, we have
\begin{align*}
    \lambda_1((s-1)K_1\vee qK_t) &= \frac{t-1+\sqrt{(t-1)^2+4(s-1)qt}}{2},\\
    \lambda_n((s-1)K_1\vee qK_t)&= \frac{t-1 - \sqrt{(t-1)^2+4(s-1)qt}}{2}.
\end{align*}
Thus $S((s-1)K_1\vee qK_t)=\sqrt{(t-1)^2+4(s-1)qt}$. Let $q=\lfloor (n-s+1)/t\rfloor$.
By the eigenvalue interlacing theorem, we then have
$$S(G)\geq \sqrt{(t-1)^2+4(s-1)qt}\geq \sqrt{4(s-1)(n-s+1)+(t-1)^2-4(s-1)(t-1)}=2\sqrt{a_0}+O\left(\frac{1}{\sqrt{n}}\right).
$$

Therefore, 
\begin{align*}
|\lambda_n(G)| &= S(G)-\lambda_1(G) \\
&\geq 2\sqrt{a_0}+O\left(\frac{1}{\sqrt{n}}\right) - \left(\sqrt{a_0} +\frac{s+t-3}{2}+O\left(\frac{1}{\sqrt{n}}\right) \right)\\ &=\sqrt{a_0} - \frac{s+t-3}{2}-O\left(\frac{1}{\sqrt{n}}\right).
\end{align*}
\end{proof}

For the rest of this paper, let ${\bf x}$ and ${\bf z}$ be the eigenvectors of $A(G)$ corresponding to the eigenvalues $\lambda_1$ and $\lambda_n$ respectively. For convenience, let ${\bf x}$ and ${\bf z}$ be indexed by the vertices of $G$. By the Perron-Frobenius theorem, we may assume that all entries of ${\bf x}$ are positive. We also assume that $\cx$ and $\cz$ are normalized so that the maximum absolute values of the entries of $\cx$ and $\cz$ are equal to $1$, and so
there are vertices $u_0$ and $w_0$ with ${\bf x}_{u_0} = |\cz_{w_0}| = 1$. 

Let $V_{+}=\{v\colon {\bf z}_v> 0\}$, $V_0=\{v\colon {\bf z}_v= 0\}$,
and $V_-=\{v\colon {\bf z}_v < 0\}$. 
Since $\cz$ is a non-zero vector, at least one of $V_{+}$ and $V_{-}$ is non-empty. By considering the eigen-equations of $\lambda_n \sum_{v\in V_{+}}\cz_v$ or $\lambda_n \sum_{v\in V_{-}}\cz_v$, we obtain that both $V_{+}$ and $V_{-}$ are non-empty.
For any vertex subset $S$, we define the \textit{volume} of $S$, denoted by $\vol(S)$, as 
$\vol(S)= \sum_{v\in S} |\cz_v|$. 
In the following lemmas, we use the bounds of $\lambda_n$ to deduce some information on $V_{+}$, $V_{-}$ and $V_0$.

\begin{lemma}
We have
$$\vol(V(G))=O(\sqrt{n}).$$
\end{lemma}
\begin{proof}
For any vertex $v \in V(G)$, we have
$$d(v) \geq |\sum_{y\in N(v)}\cz_y|=|\lambda_n| |\cz_v|.$$
Applying Theorem \ref{thm:linear-edge-bound} and Corollary \ref{cor:Mader},
we have
$$|\lambda_n| \vol(V) = \sum_{v\in V(G)}  |\lambda_n| |\cz_v|
 \leq  \sum_{v\in V(G)}d(v) =O(n).$$
By Lemma \ref{lem:lambdan}, $|\lambda_n|\geq \sqrt{(n-s+1)} - \frac{s+t-3}{2}-O\left(\frac{1}{\sqrt{n}}\right)$.
We thus have $\vol(V)=O(\sqrt{n})$.
\end{proof}

Without loss of generality, we assume $|V_+|\leq \frac{n}{2}$. 
\begin{lemma}
We have 
    \[\vol(V_+)=O(1).\]
\end{lemma}

\begin{proof}
    Let $\epsilon>0$ be a small constant depending on $s$ and $t$ to be chosen later. 
Define two sets $L$ and $S$ as follows: 
\[L=\{v\in V_+\colon |N(v)\cap V_-|\geq \epsilon n\},\]
and $S=V_+\setminus L$. Let $C=4^{s+1}s!t$. By Theorem \ref{thm:bipartite_edge_bound}, we have
\begin{equation}\label{eq:Lbound}
    |L|\leq \frac{E(L,V_-)}{\epsilon n} \leq \frac{Cn}{\epsilon n}=\frac{C}{\epsilon}.
\end{equation}

We then have that
\begin{align}
\lambda_n^2 \vol(S) &= \lambda_n^2\sum_{v\in S} \cz_v \nonumber\\
&=\lambda_n\sum_{v\in S} \sum_{u\in N(v)} \cz_u \nonumber\\
&\leq  \sum_{v\in S} \sum_{u\in N(v)\cap V_-} \lambda_n \cz_u \nonumber\\
&\leq  \sum_{v\in S} \sum_{u\in N(v)\cap V_-} \sum_{y\in V_+\cap N(u)} \cz_y \nonumber\\
&= \sum_{y\in V_+}\cz_y |E(S, N(y)\cap V_-)| \nonumber\\
&= \sum_{y\in L}\cz_y |E(S, N(y)\cap V_-)|+ \sum_{y\in S}\cz_y |E(S, N(y)\cap V_-)|. \label{eq:S0}
\end{align}
We apply the following estimation.
For $y\in L$, we have
\begin{equation}
\label{eq:S1}
  |E(S, N(y)\cap V_-)|\leq |E(S, V_-)|\leq Cn.
\end{equation}
For $y\in S$, by Theorem \ref{thm:bipartite_edge_bound}, we have
\begin{equation}\label{eq:S2}
    |E(S, N(y)\cap V_-)| \leq  (s-1)|S|+ C\epsilon n.
\end{equation}
Now we apply the assumption that $|V_+|\leq \frac{n}{2}$. We have
\begin{equation}\label{eq:S3}
    |E(S, N(y)\cap V_-)| \leq (s-1)\frac{n}{2}+ C\epsilon n.
\end{equation}
Plugging Equations \eqref{eq:S1} and \eqref{eq:S3} into Equation \eqref{eq:S0},
we get
\begin{equation} \label{eq:S4}
    \lambda_n^2 \vol(S) \leq \vol(L) Cn + \vol(S)\left((s-1)\frac{n}{2}+ C\epsilon n\right). 
\end{equation}
 By Lemma \ref{lem:lambdan}, $|\lambda_n|\geq \sqrt{(s-1)(n-s+1)} - \frac{s+t-3}{2}-O\left(\frac{1}{\sqrt{n}}\right)$.
Set $\epsilon=\frac{s-1}{6C}$.
  We have that for sufficiently large $n$,
\begin{equation}\label{eq:S5}
    \lambda_n^2 - \left((s-1)\frac{n}{2}+ C\epsilon n\right) >\frac{(s-1)n}{4}.
\end{equation}
Combining Equations \eqref{eq:S4} and \eqref{eq:S5} and solving $\vol(S)$, we get
\begin{equation}
    \vol(S)\leq \frac{4C}{s-1}\vol(L).
\end{equation}
This implies
\begin{align*}
    \vol(V_+) &\leq \left(1+ \frac{4C}{s-1}\right) \vol(L) \\
    &\leq \left(1+ \frac{4C}{s-1}\right)|L|\\
    &\leq \left(1+ \frac{4C}{s-1}\right)\frac{C}{\epsilon}\\
    &=O(1).
\end{align*}
At the last step, we apply Inequality \eqref{eq:Lbound}.
The proof of this lemma is thus finished.
\end{proof}

\begin{corollary}\label{cor:w0}
    For any $v\in V_-$, we have
    \[|\cz_v|=O\left( \frac{1}{\sqrt{n}}\right). \]
In particular, $w_0\in V_+$ and $|N(w_0)\cap V_-|=\Omega(n)$.
\end{corollary}
\begin{proof}
    For any $v\in V_-$, we have
    \[|\lambda_n| |\cz_v|=\lambda_n \cz_v\leq \sum_{y \in N(v)\cap V_+} \cz_y\leq \vol(V_+) =O(1).\]
    This implies $\cz_v= O\left( \frac{1}{\sqrt{n}}\right).$
    In particular, we have $w_0\in V_+$. Thus $\cz_{w_0}=1$.
    We then obtain that
\begin{align*}
    \lambda_n^2 &= \lambda_n^2 \cz_{w_0}\\
    &\leq \lambda_n \sum_{v\in N(w_0)\cap V_-} \cz_v\\
    &\leq |N(w_0)\cap V_-| \cdot \cz_{w_0} + \sum_{y\in V_+\setminus \{w_0\}}\cz_y |N(y)\cap N(w_0)\cap V_-|\\
    &\leq |N(w_0)\cap V_-| \vol(V_+).
\end{align*}
Since $\vol(V_+)=O(1)$ and $\lambda_n^2\geq (s-1-o(1))n$, we have
$|N(w_0)\cap V_-|=\Omega(n)$.
\end{proof}

\begin{lemma}\label{lem:Vsize}
    We have $|V_-|\geq n- O(\sqrt{n})$ and $|V_+|=O(\sqrt{n})$.
\end{lemma}
\begin{proof}
We define $L$ now as follows. Let
\[L=\{v\in V_+\colon |N(v)\cap V_-|\geq C_1\sqrt{n}\},\]
where $C_1$ is some big constant chosen later. Let $S=V_+\setminus L$.
We have
\begin{equation}\label{eq:L'}
    |L|\leq\frac{E(L, V_-)}{C_1\sqrt{n}}\leq \frac{Cn}{C_1\sqrt{n}}=\frac{C}{C_1}\sqrt{n}.
\end{equation}
By Corollary \ref{cor:w0}, we have $w_0\in L$. In particular, $\vol(L)\geq 1$.

Similar to Inequality \eqref{eq:S0}, we have
\begin{align}
\lambda_n^2 \vol(L) 
&\leq \sum_{y\in V_+}\cz_y |E(L, N(y)\cap V_-)| \nonumber\\
&= \sum_{y\in L}\cz_y |E(L, N(y)\cap V_-)|+ \sum_{y\in S}\cz_y |E(L, N(y)\cap V_-)|. \label{eq:L0}
\end{align}

We apply the following estimation.
For $y\in L$, we have
\begin{equation}
\label{eq:L1}
  |E(L, N(y)\cap V_-)|\leq |E(L, V_-)|\leq (s-1)|V_-| + C|L|.
\end{equation}
For $y\in S$, we have
\begin{equation}\label{eq:L2}
    |E(L, N(y)\cap V_-)| \leq (s-1)|L|+ CC_1\sqrt{n}.
\end{equation}
Combining Equations \eqref{eq:L0}, \eqref{eq:L1}, and \eqref{eq:L2}, we get
\begin{equation}
    \lambda_n^2 \vol(L) \leq \vol(L) \left((s-1)|V_-| + C|L|\right)
    + \vol(S)\left( (s-1)|L|+ CC_1\sqrt{n}\right).
\end{equation}
Equivalently, we have
\begin{align*}
    |V_-|&\geq \frac{\lambda_n^2}{s-1} - C |L| -\frac{\vol(S)}{\vol(L)} \left( (s-1)|L|+ CC_1\sqrt{n}\right)\\
      &\geq n- C'\sqrt{n}
\end{align*}
for some sufficiently large constant $C'$. Here we apply Inequality \eqref{eq:L'} that $\vol(S)\leq \vol(V_+)=O(1)$ and
$\vol(L)\geq 1$.
Thus, we have
$|V_+|=O(\sqrt{n})$.
\end{proof}

\begin{lemma}\label{lem:degree_structure}
There exist some constant $C_2$ and $s-1$ vertices $u_1, \ldots, u_{s-1}$ satisfying
\begin{enumerate}[(i)]
    \item For any $1\leq i\leq s-1$, we have $d(u_i)\geq n- C_2\sqrt{n}$.
    \item For any vertex $v\not\in \{u_1, \ldots, u_{s-1}\}$, we have $d(v)\leq sC_2\sqrt{n}$.
\end{enumerate}
\end{lemma}
\begin{proof}
This time we define $L$ as follows: 
\[L=\{v\in V_+\colon |N(v)\cap V_-|\geq n -C_2 \sqrt{n}\},\]
where $C_2$ is some big constant chosen later, and let $S=V_+\setminus L$.  

We first claim that $|L|\leq s-1$. Otherwise, there exist $s$ vertices $u_1,\ldots, u_s \in L$. We have
\[\displaystyle\bigcap_{i=1}^s (N(u_i)\cap V_-)\geq
|V_-|-sC_2\sqrt{n}>t,
\] when $n$ is sufficiently large.
Therefore, $G$ contains a subgraph $K_{s,t}$, giving a contradiction. 
Hence $|L|\leq s-1$. 

Now let us consider $\lambda_n^2 \vol(V_+)$. By Lemma \ref{lem:Vsize}, we know that $|V_{+}|\leq C'\sqrt{n}$ for some constant $C'$.
As before, we have
\begin{align}
\lambda_n^2 \vol(V_+)
&\leq \sum_{y\in V_+}\cz_y |E(V_+, N(y)\cap V_-)| \nonumber\\
&= \sum_{y\in L}\cz_y |E(V_+, N(y)\cap V_-)|+ \sum_{y\in S}\cz_y |E(V_+, N(y)\cap V_-)|. \label{eq:V+0a}
\end{align}
We apply the following estimation. We let $C = 4^{s+1}s!t$. 
For $y\in S$, we have
\begin{equation}
\label{eq:V+1a}
  |E(V_+, N(y)\cap V_-)|\leq (s-1) |N(y)\cap V_-| + C|V_+|
  \leq (s-1)(n-C_2\sqrt{n}) + CC'\sqrt{n}. 
\end{equation}
For $y\in L$, we have
\begin{equation}\label{eq:V+2a}
    |E(V_+, N(y)\cap V_-)| \leq |E(V_+, V_-)|\leq (s-1)n+ CC'\sqrt{n}.
\end{equation}

Plugging Equations \eqref{eq:V+1a} and \eqref{eq:V+2a} into Equation \eqref{eq:V+0a},
we get
\begin{align} \nonumber
    \lambda_n^2 \vol(V_+) &\leq \vol(S) \left((s-1)(n-C_2\sqrt{n}) + CC'\sqrt{n} \right) + 
    \vol(L)\left( (s-1)n+ CC'\sqrt{n} )\right)\\
    &=\vol(V_+)\left( (s-1)n+ CC'\sqrt{n} )\right) - \vol(S)(s-1) C_2\sqrt{n}.\label{eq:V+4}
\end{align}

Applying the lower bound of $|\lambda_n|$ in Lemma \ref{lem:lambdan}, we conclude
\begin{equation}\label{eq:volS}
\vol(S) \leq \frac{CC'+ (s-1)(s+t-3)+O(1)}{(s-1)C_2} \vol(V_+).
\end{equation}
Choose $C_2$ large enough such that $\frac{CC'+ (s-1)(s+t-3)+O(1)}{(s-1)C_2}\leq \frac{1}{s^2}$ and $C_2\sqrt{n}-|V_{+}|\geq t$ (recall that $|V_{+}|=O(\sqrt{n})$ by Lemma \ref{lem:Vsize}). We then have that
\[\vol(S)\leq \frac{1}{s^2}\vol(V_+).\]
This implies 
\[\vol(S)\leq \frac{1}{s^2-1}\vol(L).\]
Since $\vol(L)\leq |L|\leq s-1$, we get 
\[\vol(S)\leq \frac{s-1}{s^2-1}=\frac{1}{s+1}.\]

Now we do the similar calculation for $\vol(L)$. We have
\begin{align}
\lambda_n^2 \vol(L)
&\leq \sum_{y\in V_+}\cz_y |E(L, N(y)\cap V_-)| \nonumber\\
&= \sum_{y\in L}\cz_y |E(L, N(y)\cap V_-)|+ \sum_{y\in S}\cz_y |E(L, N(y)\cap V_-)|. \label{eq:V+0}
\end{align}
We apply the following estimation.
For $y\in S$, we have
\begin{equation}
\label{eq:V+1}
  |E(L, N(y)\cap V_-)|\leq (s-1) |N(y)\cap V_-)| +C|L|
  \leq (s-1)(n-C_2\sqrt{n}) + C(s-1). 
\end{equation}
For $y\in L$, we have
\begin{equation}\label{eq:V+2}
    |E(L, N(y)\cap V_-)| \leq |E(L, V_-)|\leq |L|n.
\end{equation}

Plugging Equations \eqref{eq:V+1} and \eqref{eq:V+2} into Equation \eqref{eq:V+0},
we get
\begin{equation} \label{eq:V+4}
    \lambda_n^2 \vol(L) \leq \vol(S) \left((s-1)(n-C_2\sqrt{n}) + C(s-1) \right) + 
    \vol(L) |L|n.
\end{equation}
Since $w_0\in L$, we have $\vol(L)\geq 1$. We then obtain that
\begin{align*}
    |L|&\geq \frac{\lambda_n^2}{n} - \frac{1}{(s^2-1)n} \left((s-1)(n-C_2\sqrt{n}) + C(s-1) \right)\\
    &\geq s-1 -\frac{1}{s+1} +o(1).
\end{align*}
Since $|L|$ is an integer, we have 
\[|L|\geq s-1.\]
Together with the upper bound in Inequality \eqref{eq:L'}, we get $|L|=s-1$.

Now we could write $L=\{u_1, \ldots, u_{s-1}\}$. We then have that
\begin{equation}
    \abs*{\bigcap_{i=1}^{s-1} (N(u_i)\cap V_-)}\geq |V_-|-(s-1)C_2\sqrt{n}.
\end{equation}
Now we claim that for any vertex $v\not\in L$, 
$d(v) \leq sC_2\sqrt{n}$. Otherwise, since $C_2$ is chosen such that $C_2\sqrt{n}-|V_{+}|\geq t$, we then have
$$\abs*{N(v) \cap \lp \bigcap_{i=1}^{s-1} (N(u_i)\cap V_-)\rp}\geq sC_2\sqrt{n}-|V_{+}|-(s-1)C_2\sqrt{n}\geq C_2\sqrt{n}-|V_{+}|\geq t,$$
which implies that $L\cup \{v\}$ and $t$ of their common neighbors form a $K_{s,t}$ in $G$, giving a contradiction.
Thus, $d_{v}\leq sC_2\sqrt{n}$ for any $v\notin L$.
\end{proof}

\begin{lemma} We have
    \begin{enumerate}[(i)]
    \item For any $1\leq i\leq s-1$, $\cz_{u_i}= 1- O\left(\frac{1}{\sqrt{n}}\right)$.
    \item For any vertex $v\not\in \{u_1, \ldots, u_{s-1}\}$, we have   $|\cz_v|=O(\frac{1}{\sqrt{n}})$.
\end{enumerate}
\end{lemma}
\begin{proof}
We will prove (ii) first. 
Let $C_2$ be the same constant obtained from Lemma \ref{lem:degree_structure}. Let $L=\{v\in V_+\colon |N(v)\cap V_-|\geq n -C_2 \sqrt{n}\}$, and $S=V_+\setminus L$. By Corollary \ref{cor:w0}, we know that for every $v \in V^{-}$, $|\cz_v| = O(\frac{1}{\sqrt{n}})$. Thus it suffices to show that for every $v \in S$, $|\cz_v| = O(\frac{1}{\sqrt{n}})$. Indeed, for every $v\in S$, we have that
\begin{align*}
    |\lambda_n|^2\cz_v &\leq |\lambda_n|  \sum_{u\in N(v)\cap V_-} |\cz_u|\\
      &\leq \sum_{u\in N(v)\cap V_-} \sum_{y\in N(u)\cap V_+}\cz_y\\
      &= \sum_{y\in V_+}\cz_y\cdot |N(v)\cap N(y)\cap V_-| \\
      &\leq sC_2 \cdot \sum_{y\in V_+}\cz_y \\
      &\leq sC_2 \cdot O(1).
      \end{align*}
Thus, $\cz_v= O\left(\frac{1}{n}\right)$. This completes the proof of (ii).

Finally, we estimate $\cz_{u_i}$ for $1\leq i\leq s-1$. By previous lemmas, we know that $w_0 \in \{u_1, \ldots, u_{s-1}\}$. From the eigen-equations, we obtain that for each $u_i$ ($1\leq i\leq s-1$),

\begin{align}
    |\lambda_n| (\cz_{w_0} -\cz_{u_i}) & = -\sum_{u\in N(w_0)\setminus N(u_i)} \cz_u + \sum_{u\in N(u_i)\setminus N(w_0)} \cz_u \\
    &\leq \sum_{u\in (N(w_0)\setminus N(u_i))\cap V_-} |\cz_u| +  \sum_{u\in (N(u_i)\setminus N(w_0))\cap V_+} \cz_u\\
    &\leq \sum_{u\in (N(w_0)\setminus N(u_i))\cap V_-} |\cz_u|  +  O(1)\\
    &\leq C_2\sqrt{n} \cdot O\left(\frac{1}{\sqrt{n}}\right) +  O(1) \\
    &= O(1).
\end{align}
Therefore, we have
$\cz_{u_i}\geq 1 - O\left(\frac{1}{\sqrt{n}}\right)$ since $\cz_{w_0} = 1$ and $\cz_{w_0}-\cz_{u_i} = O(\frac{1}{\sqrt{n}})$.
\end{proof}

Recall that we let $L:= \{u_1, u_2, \cdots, u_{s-1}\}$.
Let $V' = \{v\in V(G)\backslash L: |N(v) \cap L| = s-1\}$ and let $V'' = V(G)\backslash (L\cup V')$. We have the following lemma on the structure of $G$.

\begin{lemma}\label{lem:struc}
We have the following properties.
\begin{enumerate}[(i)]
    \item $|V'| \geq n- (s-1)C_2\sqrt{n}$.
    \item For any $v\in V(G)\backslash L$, $|N(v)\cap V'| \leq t-1$.
    \item In $H= G[V(G)\backslash L]$, for any vertex $v\in V(H)$, $|N_H(N_H(v)) \cap V'| \leq t^2$.
\end{enumerate}
\end{lemma}
\begin{proof}
By Lemma \ref{lem:degree_structure}, $\displaystyle\min_{u\in L} d(u)\geq n-C_2\sqrt{n}$. It follows that $|V'|\geq n-(s-1)C_2\sqrt{n}\geq t$. For any $v\in V(G)\backslash L$, $v$ has at most $t-1$ neighbors in $V'$, otherwise, $L\cup \{v\}$ and $t$ of their common neighbors in $V'$ would form a $K_{s,t}$ in $G$.

Now for any $v \in V(G)\backslash L$, we claim that $|N_{H}(N_{H}(v))\cap V'|\leq t^2$. Indeed, suppose not, then by (ii) and the Pigeonhole principle, there exist $t$ vertex-disjoint $2$-vertex paths from $v$ to $t$ distinct vertices in $V'$. But then it is not hard to see that $L\cup \{v\}$ and these $t$ distinct vertices would form a $K_{s,t}$ minor, giving a contradiction. 
\end{proof}

\begin{lemma} We have
    \begin{enumerate}[(i)]
    \item For any $1\leq i\leq s-1$, $\cx_{u_i}= 1- O\left(\frac{1}{\sqrt{n}}\right)$.
    \item For any vertex $v\not\in \{u_1, \ldots, u_{s-1}\}$, we have $\cx_v=O(\frac{1}{\sqrt{n}})$.
\end{enumerate}
\end{lemma}
\begin{proof}
    Let us prove (ii) first. For any vertex $v\notin\{u_1, \ldots, u_{s-1}\}$, by the eigen-equations, we have that

\begin{align*}
    \lambda_1^2 \cx_v &= \lambda_1 \sum_{u\in N(v)}\cx_u \\
    &= \lambda_1 \left(\sum_{u\in N(v)\cap V'}\cx_u + \sum_{u\in N(v)\cap L} \cx_u + \sum_{u\in N(v)\cap V''} \cx_u     \right)\\ 
    &\le \lambda_1 \left((t-1) + (s-1) + \sum_{u\in N(v)\cap V''} \cx_u     \right)\\ 
    &= (t+s-2)\lambda_1 +  \sum_{u\in N(v)\cap V''} \lambda_1 \cx_u\\
      &=  (t+s-2)\lambda_1 + \sum_{u\in N(v)\cap V''} \sum_{w\in N(u)}\cx_w\\
       &=  (t+s-2)\lambda_1 + \sum_{u\in N(v)\cap V''} \left(\sum_{w\in N(u)\cap L}\cx_w +
       \sum_{w\in N(u)\cap V'}\cx_w +\sum_{w\in N(u)\cap V''}\cx_w 
       \right)\\
    &\leq   (t+s-2)\lambda_1 +  (s-1)|V''| + t^2 + \sum_{u\in N(v)\cap V''} \sum_{w\in N(u)\cap V''}\cx_w\\
    &\leq  (t+s-2)\lambda_1 +  (s-1)|V''| + t^2 + 2|E(G[V''])|\\
     &\leq (t+s-2)\lambda_1 +  (s-1)(sC_2\sqrt{n}) + t^2 + O(\sqrt{n})\\
     &= O(\sqrt{n}).
\end{align*}
It follows that $\cx_v = O(\frac{1}{\sqrt{n}})$.

Now we will prove (i). Recall that $u_0$ is a vertex such that $\cx_{u_0} = 1$. Thus $u_0 \in L$. 
Let $u_i$ be an arbitrary vertex in $L\backslash \{u_0\}$.

If $u_0 u_i$ is not an edge of $G$, then we have
\begin{align*}
    \lambda_1|\cx_{u_0}-\cx_{u_i}| &\leq \sum_{v\in V'' }\cx_v + \sum_{v\in L}\cx_v \\
    &\leq |V''|\cdot O\left( \frac{1}{\sqrt{n}}\right) + (s-1)\\
    &= O(1).
\end{align*}
If $u_0 u_i$ is an edge of $G$, we have
\begin{align*}
    (\lambda_1-1)|\cx_{u_0}-\cx_{u_i}| &\leq \sum_{v\in V''}\cx_v + \sum_{v\in L}\cx_v \\
    &\leq |V''|\cdot O\left( \frac{1}{\sqrt{n}}\right) + (s-1) \\
    &= O(1).
\end{align*}
In both cases, we have
$$ |\cx_{u_0}-\cx_{u_i}| =   O\left( \frac{1}{\sqrt{n}}\right).$$  
It follows that $\cx_{u_i} = 1-O\lp \frac{1}{\sqrt{n}}\rp$ for any $i\in [s-1]$.
\end{proof}

Now we are ready to show that $G$ has $s-1$ vertices that are connected to each of the rest of the $n-s+1$ vertices.

\begin{lemma}\label{lem:compbipsubgraph}
    $G$ contains the subgraph $K_{s-1, n-s+1}$.
\end{lemma}
\begin{proof}
      Let $\cx$ and $\cz$ be the eigenvectors associated with $\lambda_1$ and $\lambda_n$ respectively. Assume that $\cx$ and $\cz$ are both normalized such that the largest entries of them in absolute value are $1$.
    By Lemma \ref{lem:degree_structure}, there exist $s-1$ vertices $L = \{u_1, u_2, \cdots, u_{s-1}\}$ such that for every $v\in L$, $d(v) \geq n -C_2\sqrt{n}$ and for every $v\notin L$, $d(v)\leq sC_2\sqrt{n}$. Recall that $V' = \{v\in V(G)\backslash L: |N(v) \cap L| = s-1\}$ and $V'' = V(G)\backslash (L\cup V')$.
    
    To prove the lemma, it suffices to show that $V''$ is empty. Suppose otherwise that $V''$ is not empty. Note that $V''$ induces a $K_{s,t}$-minor-free graph, and by Theorem \ref{thm:linear-edge-bound} and its corollary, we know that there exists some constant $C_0$ and some vertex $v_0\in V''$ such that $d_{G[V'']}(v_0) \leq C_0$.
    Moreover, observe that $v_0$ has at most $(t-1)$ neighbors in $V'$, as otherwise $L \cup \{v_0\}$ and $t$ of their common neighbors would form a $K_{s,t}$ in $G$.
    
    Let $G'$ be obtained from $G$ by removing all the edges of $G$ incident with $v_0$ and adding an edge from $v_0$ to every vertex of $L$, so that $E(G') = E(G-v_0) \cup \{v_0 u_1, v_0 u_2, \cdots, v_0 u_{s-1}\}$.  Observe $G'$ is still $K_{s,t}$-minor-free.
    
    We claim that $\lambda_n(G') < \lambda_n(G)$. Indeed, consider the vector $\tilde{\cz}$ such that $\tilde{\cz}_{u} = \cz_u$ for $u\neq v_0$ and $\tilde{\cz}_{v_0} = -|\cz_{v_0}|$. Then
    \begin{align*}
        \tilde{\cz}' A(G') \tilde{\cz} & = \dss_{uv\in E(G-v_0)}2\cz_u \cz_v + 2\tilde{\cz}_{v_0}\cdot \vol(L) \\
        & = \dss_{uv\in E(G)}2\cz_u\cz_v - 2\dss_{y\sim v_0}\cz_y\cz_{v_0} - 2|\cz_{v_0}|\vol(L) \\
        &\leq \cz' A(G) \cz + 2 \dss_{y\sim v_0} |\cz_y\cz_{v_0}| - 2|\cz_{v_0}| \cdot \vol(L)\\
                      & \leq \cz' A(G)\cz + 2 \cdot (t+C_0) \cdot O\left(\frac{1}{\sqrt{n}}\right)\cdot |\cz_{v_0}| -  2\lp 1- O\left(\frac{1}{\sqrt{n}}\right)\rp |\cz_{v_0}|\\
                      & < \cz' A(G) \cz.
    \end{align*}   

   Similarly, we claim that $\lambda_1(G') > \lambda_1(G)$. Indeed,
     \begin{align*}
       \cx'\cx \lambda_1(G') &= \cx' A(G') \cx \\
                      &= \cx' A(G) \cx - 2 \dss_{y\sim {v_0}} \cx_y \cx_{v_0} + 2\cx_{v_0} \vol(L)\\
                      & \geq \cx' \cx \lambda_1(G) - 2 \cdot (t+C_0) \cdot O\left(\frac{1}{\sqrt{n}}\right)\cdot \cx_{v_0} +  2\lp 1- O\left(\frac{1}{\sqrt{n}}\right)\rp \cx_{v_0}\\
                      & > \cx' \cx \lambda_1(G).
    \end{align*}
Hence we have $S(G') =\lambda_1(G') -\lambda_n(G') > \lambda_1(G) -\lambda_n(G) = S(G)$, giving a contradiction.
\end{proof}

\section{Proof of Theorem \ref{thm:main}}\label{sec:main_thm}

By Lemma~\ref{lem:compbipsubgraph}, a maximum-spread $K_{s,t}$-minor-free graph $G$ contains a complete bipartite subgraph $K_{s-1, n-s+1}$. We denote the part of $s-1$ vertices by $L$ and the other part of $n-s+1$ vertices by $R$.  
Let $\alpha$ be a normalized eigenvector corresponding to an eigenvalue $\lambda$ of the adjacency matrix of $G$. Let $A_L$ (or $A_R$) be the adjacency matrix of the induced subgraph $G[L]$ (or $G[R]$) respectively.

Let $\alpha_L$ (respectively, $ \alpha_R$)
denote the restriction of $\alpha$ to $L$ (respectively, $R$).
The following lemma computes the vectors $\alpha_L$ and $\alpha_R$.

\begin{lemma}\label{l1}
  If $|\lambda| > t - 1$, then
    \begin{align} \label{eq:taylor_eqR} 
       \alpha_R =(\one' \alpha_L)\sum_{k=0}^\infty \lambda^{-(k+1)} A_R^k \one,\\
        \label{eq:taylor_eqL} 
        \alpha_L =(\one' \alpha_R)\sum_{k=0}^\infty \lambda^{-(k+1)} A_L^k \one.
    \end{align}
\end{lemma}
\begin{proof}
Note that since $G$ is $K_{s,t}$-minor-free and every vertex in $L$ is adjacent to every vertex in $R$, it follows that $G[R]$ is $K_{1,t}$-minor-free, and thus the maximum degree of $G[R]$ is at most $t-1$.  For $n$ sufficiently large,
both $\lambda_1(G)$ and $|\lambda_n(G)|$ are greater than $t-1$. 
Hence when restricting the coordinates of $A(G)\alpha$ to $R$, we have that
\begin{equation}\label{eq:bx}
    A_R\alpha_R + (\one'\alpha_L) \one = \lambda \alpha_R.
\end{equation}
It then follows that
\begin{align}
    \alpha_R &= (\one'\alpha_L) (\lambda I-A_R)^{-1}\one  \nonumber\\
    &=(\one'\alpha_L) \lambda^{-1}  (I-\lambda^{-1}A_R)^{-1}\one  \nonumber\\
    &= (\one'\alpha_L)\lambda^{-1} \sum_{k=0}^\infty (\lambda^{-1}A_R)^{k} \one\nonumber\\
    &= (\one'\alpha_L)\sum_{k=0}^\infty \lambda^{-(k+1)} A_R^k \one. 
    \end{align}
Here we use the assumption that $|\lambda|> t-1 \geq \lambda_1(A_R)$ so that the infinite series converges.
Similarly, we have
$$\alpha_L =(\one' \alpha_R)\sum_{k=0}^\infty \lambda^{-(k+1)} A_L^k \one.$$
\end{proof}

\begin{lemma}\label{l:lambda}
Both $\lambda_1$ and $\lambda_n$ satisfy the following equation.
\begin{equation} \label{eq:lambda}
        \lambda^2 = 
  \left(\sum_{k=0}^\infty \lambda^{-k} \one' A_L^k \one\right) \cdot
   \left(\sum_{k=0}^\infty \lambda^{-k} \one' A_R^k \one\right).
\end{equation}
\end{lemma}
\begin{proof}From Equations \eqref{eq:taylor_eqR} and \eqref{eq:taylor_eqL}, we have
    \begin{align}
       \one'\alpha_R =(\one' \alpha_L)\sum_{k=0}^\infty  \one' \lambda^{-(k+1)} A_R^k \one,\\
        \one'\alpha_L =(\one' \alpha_R)\sum_{k=0}^\infty \one'\lambda^{-(k+1)} A_L^k \one.
    \end{align}
Thus
\begin{equation}
\one'\alpha_R =(\one' \alpha_R)
\left(\sum_{k=0}^\infty \one'\lambda^{-(k+1)} A_L^k \one\right) \cdot
\left(\sum_{k=0}^\infty \lambda^{-(k+1)} \one' A_R^k \one\right).
\end{equation}
Since $\one'\alpha_R>0$, equation \eqref{eq:lambda} is obtained by canceling $\one'\alpha_R$. 
\end{proof}

For $k=1,2,3\ldots$, let $l_k=  \one' A_L^k \one$, $r_k=  \one' A_R^k \one$, and $a_k=\sum_{j=0}^kl_jr_{k-j}$.
Then Equation \eqref{eq:lambda} can be written as:
\begin{equation}\label{eq:lambda_a}
    \lambda^2=\sum_{k=0}^\infty a_k \lambda^{-k}.
\end{equation}

In particular, we have
\begin{align}
l_0 &= s-1; \\
l_1 &= 2|E(G[L])|; \\
r_0 &= n-s+1; \\
r_1 &= 2|E(G[R])|; \\
a_0 &=l_0r_0=(s-1)(n-s+1),\\
a_1&=l_0r_1+l_1r_0.
\end{align}

\begin{lemma} \label{lem:approximation}
We have the following estimation on the spread of $G$:
\begin{equation}\label{eq:spreadapprox}
    S(G)=2\sqrt{a_0}+\frac{2c_2}{\sqrt{a_0}} + \frac{2c_4}{a_0^{3/2}} +  \frac{2c_6}{a_0^{5/2}} + O\left(a_0^{-7/2}\right).
\end{equation}
Here 
\begin{align}
a_0 &=(s-1)(n-s+1)\\
c_2 &= -\frac38 \left(\frac{a_1}{a_0}\right)^2 + \frac12 \frac{a_2}{a_0}, \label{eq:c2}
\\
c_4 &=-\frac{105}{128} \left(\frac{a_1}{a_0}\right)^4 +\frac{35}{16} \left(\frac{a_1}{a_0}\right)^2\frac{a_2}{a_0}
-\frac{5}{8}\left(\frac{a_2}{a_0}\right)^2 -\frac{5}{4}\frac{a_1}{a_0}\frac{a_3}{a_0} +\frac{1}{2} \frac{a_4}{a_0} \label{eq:c4}\\
c_6&=-\frac{3003}{1024} \left(\frac{a_1}{a_0}\right)^6 +\frac{3003}{256} \left(\frac{a_1}{a_0}\right)^4\frac{a_2}{a_0}
-\frac{693}{64} \left(\frac{a_1}{a_0}\right)^2\left(\frac{a_2}{a_0}\right)^2
+\frac{21}{16}\left(\frac{a_2}{a_0}\right)^3  
\nonumber\\
&\hspace*{5mm} 
-\frac{21}{32}\left(11\left(\frac{a_1}{a_0}\right)^3
-12\left(\frac{a_1}{a_0}\right)\left(\frac{a_2}{a_0}\right)
\right )\left(\frac{a_3}{a_0}\right)
- \frac{7}{8}
\left(\frac{a_3}{a_0}\right)^2 \nonumber\\
&\hspace*{5mm} 
+
\frac{7}{16}\left(9\left(\frac{a_1}{a_0}\right)^2 -4\frac{a_2}{a_0}\right)\frac{a_4}{a_0}
- \frac{7}{4}\frac{a_1}{a_0}\frac{a_5}{a_0} +\frac{1}{2} \frac{a_6}{a_0}. \label{eq:c6}
\end{align}
\end{lemma}
\begin{proof}
    Recall that by \eqref{eq:lambda_a}, we have that for $\lambda \in \{\lambda_1, \lambda_n\}$,
    $$ \lambda^2 =  a_0 +
  \sum_{k=1}^\infty a_k\lambda^{-k}.  $$
  By the main lemma in the appendix of \cite{LLLW2022}, $\lambda$ has the following series expansion:
 $$\lambda_1 = \sqrt{a_0} + c_1 + \frac{c_2}{\sqrt{a_0}} + \frac{c_3}{a_0} + \frac{c_4}{a_0^{3/2}} + \frac{c_5}{a_0^2} + \frac{c_6}{a_0^{5/2}} + O\left(a_0^{-7/2}\right).$$
Similarly, 
  $$\lambda_n = -\sqrt{a_0} + c_1 - \frac{c_2}{\sqrt{a_0}} + \frac{c_3}{a_0} - \frac{c_4}{a_0^{3/2}} + \frac{c_5}{a_0^2} - \frac{c_6}{a_0^{5/2}} + O\left(a_0^{-7/2}\right).$$

Using SageMath, we get that $c_2, c_4, c_6$ are the values in Equations \eqref{eq:c2}, \eqref{eq:c4}, \eqref{eq:c6} respectively.
It follows that
$$S(G) =\lambda_1 - \lambda_n =  2\sqrt{a_0}+\frac{2c_2}{\sqrt{a_0}} + \frac{2c_4}{a_0^{3/2}} +  \frac{2c_6}{a_0^{5/2}} + O\left(a_0^{-7/2}\right).$$
\end{proof}

\begin{proof}[Proof of Theorem~\ref{thm:main}]
     Recall that by Lemma \ref{lem:approximation}, we have the following estimation of the spread of $G$:
\begin{equation}
    S(G)=2\sqrt{a_0}+\frac{2c_2}{\sqrt{a_0}} + \frac{2c_4}{(n-1)^{3/2}} +  \frac{2c_6}{(n-1)^{5/2}} + O\left(n^{-7/2}\right).
\end{equation}
where $c_2$, $c_4$ and $c_6$ are as in Lemma~\ref{lem:approximation}. 

Since $G$ is $K_{s,t}$-minor free, $G[R]$ is $K_{1,t}$-minor free. Thus the maximum degree of $G[R]$ is at most $t-1$. In particular,
$r_2\leq (t-1)r_1$. All $c_i$'s are bounded by constants depending on $t$.
Note that
\begin{align*}
    c_2 &= -\frac38 \left(\frac{a_1}{a_0}\right)^2 + \frac12 \frac{a_2}{a_0}\\
        &=  -\frac38 \left(\frac{l_1r_0+l_0r_1}{l_0r_0}\right)^2 + \frac12 \frac{l_2r_0+l_1r_1+l_0r_2}{r_0l_0}\\
         &=   -\frac38 \left(\frac{l_1}{l_0} +\frac{r_1}{r_0}\right)^2 + \frac12 \left(\frac{l_2}{l_0} +\frac{l_1}{l_0}\frac{r_1}{r_0}+\frac{r_2}{r_0}\right)\\
         &=-\frac38 \left(\frac{l_1}{3l_0} +\frac{r_1}{r_0}\right)^2 + \frac{l_2}{2l_0} - \frac{l_1^2}{3l_0^2} +\frac{r_2}{2r_0}\\
         &= \frac{(t-1)^2}{6}  -\frac38 \left(\frac{l_1}{3l_0} +\frac{r_1}{r_0}-\frac{2}{3}(t-1)\right)^2 + \frac{l_2}{2l_0}-\frac{(t-1)l_1}{6l_0} - \frac{l_1^2}{3l_0^2}
         +\frac{r_2-(t-1)r_1}{2r_0}\\
          &= \frac{(t-1)^2}{6}  -\frac38 \left(\frac{l_1}{3l_0}+\frac{r_1}{r_0}-\frac{2}{3}(t-1)\right)^2
          + \frac{\psi(G[L])}{6l_0}
         +\frac{r_2-(t-1)r_1}{2r_0} \\
       &\leq \frac{(t-1)^2}{6}+  \frac{\psi(L_{max})}{6l_0}.
        \end{align*}
At the last step, the equality holds only if
\begin{enumerate}
    \item $\psi(L)=\psi(L_{max})$,
    \item $r_2=(t-1)r_1$,
    \item $\frac{l_1}{3l_0}+\frac{r_1}{r_0}-\frac{2}{3}(t-1)=0$.
\end{enumerate} 

Thus, we have $$S(G)\leq  2\sqrt{a_0} + \frac{(t-1)^2+ \psi(L_{max})/(s-1)}{3\sqrt{a_0}} + O\lp\frac{1}{n^{3/2}}\rp.$$
This upper bound is asymptotically tight. 
Consider $G_0=L_{max}\vee  \left( \ell_0 K_t \cup \left(n-s+1- t\ell_0\right)  P_1\right)$
where $\ell_0$ is an integer such that  $\frac{l_1}{3l_0}+\frac{r_1}{r_0}-\frac{2}{3}(t-1)$ is close to zero.
Thus
$$S(G_0) = 2\sqrt{a_0} + \frac{(t-1)^2+ \psi(L_{max})/(s-1)}{3\sqrt{a_0}} + O\lp\frac{1}{n^{3/2}}\rp.$$

\begin{claim}
$G[L]=L_{max}$.
\end{claim}
\begin{proof} 
Otherwise, we have $\psi(G[L])<\psi(L_{max})$. Also, by the definition of $\psi$, we have that $\psi(G[L])\leq \psi(L_{max})-\frac{1}{s-1}$. It then follows that for sufficiently large $n$,
\begin{align*}
S(G)&\leq  2\sqrt{a_0} + \frac{(t-1)^2+ \psi(G[L])/(s-1)}{3\sqrt{a_0}} + O\lp\frac{1}{n^{3/2}}\rp\\
&<  2\sqrt{a_0} + \frac{(t-1)^2+ \psi(L_{max})/(s-1)}{3\sqrt{a_0}} + O\lp\frac{1}{n^{3/2}}\rp \\
&=S(G_0),
\end{align*}
giving a contradiction.
\end{proof}

\begin{claim}\label{cl:Asmall}
There is a constant $C$ such that the value of $\frac{l_1}{3l_0} +\frac{r_1}{r_0}$ that maximizes $S(G)$ lies in the interval $\lp\frac{2}{3}(t-1)- Cn^{-1/2}, \frac{2}{3}(t-1)+ Cn^{-1/2}\rp$.
\end{claim}
\begin{proof}
Otherwise, for any $\frac{l_1}{3l_0} +\frac{r_1}{r_0}$ not in this interval (where $C$ is chosen later), we have
\begin{align*}
    c_2 & \leq \frac{(t-1)^2}{6} - \frac{3}{8} C^2 n^{-1}+ \frac{\psi(L_{max})}{6l_0}
         +\frac{r_2-(t-1)r_1}{2r_0} \\
        & \leq \frac{(t-1)^2}{6} - \frac{3}{8} C^2 n^{-1}+ \frac{\psi(L_{max})}{6l_0}.
\end{align*}
This implies that 
$$S(G) -S(G_0) \leq -\frac{\frac{3}{4} C^2 n^{-1}}{\sqrt{a_0}} + O\lp \frac{1}{n^{3/2}}\rp < 0,$$
giving a contradiction when $C$ is chosen to be large enough such that
$- \frac{\frac{3}{4} C^2 n^{-1}}{\sqrt{a_0}} + O\lp \frac{1}{n^{3/2}}\rp < 0.$
\end{proof}

Hence from now on, we assume that $\frac{l_1}{3l_0} +
 \frac{r_1}{r_0}\in \lp\frac{2}{3}(t-1)- Cn^{1/2}, \frac{2}{3}(t-1)+ Cn^{1/2}\rp$.

\begin{claim}\label{cl:Asmall}
There is a constant $C_2$ such that the value of $r_2$ lies in the interval $[(t-1)r_1-C_2, (t-1)r_1]$.
\end{claim}
\begin{proof}
Otherwise, we assume $r_2<(t-1)r_1-C_2$ for some $C_2$ chosen later. We then have
$$
S(G)\leq  2\sqrt{a_0} + \frac{(t-1)^2+ \psi(L_{max})/(s-1)}{3\sqrt{a_0}} 
-\frac{C_2}{r_0\sqrt{a_0}}
+ O\lp\frac{1}{n^{3/2}}\rp< S(G_0),
$$
when $C_2$ is chosen to be sufficiently large, giving a contradiction.
\end{proof}

\begin{claim}\label{cl:Asmall}
For $i\geq 3$, we have $r_i\in [(t-1)^{i-1}(r_1-(i-1)C_2),r_1(t-1)^{i-1}] $.
\end{claim}
\begin{proof}
Let $R'$ be the set of vertices in $R$ such that its degree is in the interval $[1,t-2]$. We have
$$C_2\geq (t-1)r_1-r_2=\sum_{v\in R'}(t-1-d(v))d(v)\geq (t-2)|R'|.$$
This implies 
$$|R'|\leq \frac{C_2}{t-2}.$$
We have
\begin{equation}
    (t-1)r_{i-1}-r_i\leq  (t-2)|R'|(t-1)^{i-1} \leq C_2(t-1)^{i-1}.
\end{equation}
Thus,
\begin{align*}
    r_i&\geq (t-1)r_{i-1} -C_2(t-1)^{i-1}\\
     &\geq (t-1)((t-1)r_{i-2} -C_2(t-1)^{i-2}) -C_2(t-1)^{i-1} \hspace*{2cm}\mbox{ by induction}\\
     &=(t-1)^2r_{i-2}-2C_2(t-1)^{i-1}\\
     &\geq \cdots\\
     &\geq (t-1)^{i-1}r_{1} -(i-1)C_2(t-1)^{i-1}.
\end{align*}
\end{proof}

\begin{claim}\label{cl:a1a2}
    $r_2 = (t-1) r_1$ and $\frac{l_1}{3l_0}+\frac{r_1}{r_0}-\frac{2}{3}(t-1)=O(n^{-(1-\epsilon)})$ for any given $\epsilon>0$.
\end{claim}
\begin{proof}
Assume that $\frac{l_1}{3l_0}+\frac{r_1}{r_0}-\frac{2}{3}(t-1) = A$, and $r_2 = (t-1)r_1-B$, where $A \in [-Cn^{-1/2}, Cn^{-1/2}]$ and $0 \leq B\leq C_2$.
It follows that 
\begin{align*}
    c_2(G) -c_2(G_0)  &= O\lp n^{-2}\rp
    - \frac{3A^2}{8} - \frac{B}{2r_0},
\end{align*}    
and
\begin{align*}
     c_4(G) - c_4(G_0) =O(n^{-1/2}).
\end{align*}
Thus 
\begin{align*}
    S(G) -S(G_0) &= 2\frac{c_2(G)-c_2(G_0)}{\sqrt{a_0}}
+ 2\frac{c_4(G)-c_4(G_0)}{a_0^{3/2}} +\lp a_0^{-5/2}\rp\\
&\leq 2\frac{ O\lp n^{-2}\rp
    - \frac{3A^2}{8} - \frac{B}{2r_0}}{\sqrt{a_0}}
+ 2\frac{O(n^{-1/2})}{a_0^{3/2}} +\lp a_0^{-5/2}\rp.
\end{align*}
This implies $B=0$ and $A=O(n^{-3/4})$.

Notice that $A=O(n^{-3/4})$ implies $c_4(G)-c_4(G_0)=O(n^{-1/4})$, which implies
$A=O(n^{-7/8})$. Iterate this process finitely many times. We get
$A=O(n^{-(1-\epsilon)})$ for any given $\epsilon>0$.
\end{proof}

\begin{claim}
    $G[R]$ is the union of vertex disjoint $K_t$s and isolated vertices. 
\end{claim}
\begin{proof}
     Recall that $r_1= \one' A_R\one=\sum_{i\in R} d_{G[R]}(i)=2|E(G[R])|$, and $r_2= \one' A_R^2 \one=\sum_{i\in R} d_{G[R]}(i)^2$. By Claim \ref{cl:a1a2}, we have that
    $$\sum_{i\in R} d_{G[R]}(i)^2 = (t-1)\sum_{i\in R} d_{G[R]}(i).$$

Since $d_{G[R]}(v) \leq t-1$ for every $v\in R$, it follows that $G[R]$ is the disjoint union of $(t-1)$-regular graphs and isolated vertices. Let $K$ be an arbitrary non-trivial component of $G[R]$. We will show that $K$ is a clique on $t$ vertices.

For any $u, v\in V(K)$ with $uv \notin E(K)$, we claim that $|N_K(u)\cap N_K(v)| \geq t-2$. Otherwise, $|N_K(u)\backslash N_K(v)| \geq 2$ and $|N_K(v)\backslash N_K(u)|\geq 2$. Pick an arbitrary vertex $w \in N_K(u) \cap N_K(v)$ and contract $uw$ and $wv$, we then obtain a $K_{1,t}$-minor in $K$, and thus a $K_{s,t}$-minor in $G$. Similarly, for any $u,v\in V(K)$ with $uv\in E(K)$, we have $|N_K(u)\cap N_K(v)| \geq t-3$.

We claim now that for any $u,v \in V(K)$, $|N_K(u) \cap N_K(v)| \leq t-2$. Suppose otherwise that there exist vertices $u,v \in V(K)$ such that $|N_K(u) \cap N_K(v)| =t-1$. Let $w$ be an arbitrary vertex in $L$. Note that when $n$ is sufficiently large, we could find a length-two path from $w$ to each vertex in $L\backslash \{w\}$ using distinct vertices in $R\backslash V(K)$ as the internal vertices of these paths. It follows that 
$(L\backslash \{w\}) \cup \{u,v\}$ and $(N_K(u)\cap N_K(v))\cup \{w\}$ would form a $K_{s,t}$-minor in $G$.

Hence, we have that for any $u,v\in V(K)$ with $uv\notin E(K)$, $|N_K(u)\cap N_K(v)| = t-2$.
It then follows that there exist $u', v' \in V(K)$ such that $u'\in N_K(u)\backslash N_K(v)$ and $v' \in N_K(v) \backslash N_K(u)$.
We claim that $u'v' \notin E(K)$. Indeed, if $u'v'\in E(K)$, we could contract $v'u'$ into $w'$ and obtain a $K_{s,t}$ minor the same way as above.

Now note that since $u'v \notin E(K)$, we have $|N_K(u')\cap N_K(v)| = t-2$. It follows that $N_K(u')\cap N_K(v) = N_K(u)\cap N_K(v)$. Similarly, $N_K(v') \cap N_K(u) = N_K(u)\cap N_K(v)$. We will now analyze $N_K(u)\cap N_K(v)$. 

Let $G_1 = G[N_K(u) \cap N_K(v)]$. Note that for each vertex $w \in V(G_1)$, $w$ must have at most two non-neighbors in $G_1$, otherwise $|N_K(u)\cap N_K(w)| \leq t-4$, giving a contradiction. Moreover, each vertex $w \in V(G_1)$ has at least two non-neighbors in $G_1$, otherwise $d_K(w) \geq t-4+4 = t$, giving a contradiction. It follows that each vertex in $G_1$ has exactly two non-neighbors in $G_1$.

Now, if $G_1$ is a clique, we could easily find a $K_{1,t}$ in $K$ (by identifying one of the vertices in $N(u)\cap N(v)$ as the center). Hence together with $L$, we have a $K_{s,t}$-minor in $G$. Otherwise, we find $a, b\in G_1$ such that $ab \notin E(K)$. Since $|N_K(a)\cap N_K(b)|=t-2$ and each of $a$ and $b$ has exactly one non-neighbor in $G_1$, 
we then obtain that $|N_{G_1}(a)\cap N_{G_1}(b)| = t-6$, and there exist $a', b' \in V(G_1)$ such that $a' \in N_{G_1}(a)\backslash N_{G_1}(b)$ and $b'\in N_{G_1}(b)\backslash N_{G_1}(a)$. Similar to before, we have $a'b'\notin E(K)$, and $a',b'$ is each adjacent to $N_{G_1}(a)\cap N_{G_1}(b)$. Repeat this process, eventually, this process has to terminate, and we will have a $K_{1,t}$-minor in $K$, thus a $K_{s,t}$ minor in $G$.   
\end{proof}

This completes the proof of Theorem \ref{thm:main}.
\end{proof}

We now determine the maximum spread $K_{s, t}$-minor-free graphs for all admissible pairs $(s, t)$. 

\begin{proof}[Proof of Theorem \ref{thm:admissible} ]
Since $(s,t)$ is admissible, we have $G[L]=(s-1)K_1$.
We only need to consider the graph  $G_\ell = (s-1)K_1\vee \left(\ell K_t \cup (n-s+1-\ell t)P_1 \right)$.
We have $l_0=(s-1)$ and $l_i=0$ for $i\geq 1$.
We have $r_0=(n-s+1)$, and $r_i=\ell t(t-1)^i$ 
for each $i\geq 1$. 

Now we apply Lemma~\ref{l:lambda} to simplify the equation satisfied by 
both $\lambda_1$ and $\lambda_n$.
Equation \eqref{eq:lambda} can be simplified as 
\begin{align*}
    \lambda^2 &= \left(\sum_{k=0}^\infty \lambda^{-k} \one' A_L^k \one\right) \cdot
   \left(\sum_{k=0}^\infty \lambda^{-k} \one' A_R^k \one\right)\\
   &=(s-1)\left(n-s+1 +\sum_{k=1}^\infty \lambda^{-k} \ell t(t-1)^{k}\right)\\
   &=(s-1)\left(n-s+1 + \frac{\ell t \frac{t-1}{\lambda}}{1-\frac{t-1}{\lambda}}\right)\\
   &=\frac{(s-1)\left((n-s+1)\lambda-(n-s+1-\ell t) (t-1)\right)}{\lambda-(t-1)}.
\end{align*}
Simplifying it, we get  the following cubic equation:
\begin{equation}
\lambda^3 -(t-1)\lambda^2 -(s-1)(n-s+1)\lambda+ (s-1)(t-1)(n-s+1-\ell t )=0.
\end{equation}
Now let $\lambda=x+\frac{t-1}{3}$. We get the following reduced cubic equation
$\phi(x) = 0$.
\begin{equation}
x^3 - \left((s-1)(n-s+1) + \frac{1}{3}(t-1)^2\right)x+ 
(s-1)(t-1)\left( \frac{2}{3}(n-s+1)-\ell t \right) -\frac{2}{27}(t-1)^3=0.
\end{equation}
Let $p=(s-1)(n-s+1) + \frac{1}{3}(t-1)^2$ and $q=(s-1)(t-1)\left( \frac{2}{3}(n-s+1)-\ell t \right) -\frac{2}{27}(t-1)^3$. Since $\phi(x)$ has at least two real roots, 
we know from number theory that $p^3\geq \frac{27}{4}q^2$. 

We now need a lemma on the spread of a cubic polynomial. If $f$ is a cubic polynomial with three real roots, then the \textit{spread} $S(f)$ is defined to be the difference between the largest and smallest roots of $f$. 

\begin{lemma}\label{l:cubic}
Assume $p^3>\frac{27}{4}q^2$. Let $S(q)$ (with $p$ fixed) be the spread of the cubic equation
\begin{equation}\label{eq:cubic}
x^3-px+q=0.    
\end{equation}
If $2\left(\frac{p}{3}\right)^{3/2}> |q_1|>|q_2|$, then
$$S(q_1)<S(q_2).$$
\end{lemma}
Before we give the proof of Lemma~\ref{l:cubic}, we complete the proof of Theorem~\ref{thm:admissible} using Lemma~\ref{l:cubic}. 

Applying Lemma~\ref{l:cubic}, we conclude
that $S(G_\ell)=S(\phi)$ reaches the maximum if and only if $|q|$ reaches the minimum. Let $\ell_1$ be the real root of the equation
$q=0$. We have
$$\ell_1=\frac{2}{3t}\left(n-s+1 -\frac{(t-1)^2}{9(s-1)}\right).$$
Since $q$ is a linear function on $\ell$, the function $|q|$ reaches the minimum at
the nearest integer of $\ell_1$. This completes the proof of Theorem~\ref{thm:admissible}. 
\end{proof}
We now give the proof of Lemma~\ref{l:cubic}.
\begin{proof}[Proof of Lemma~\ref{l:cubic}]
Since $p^3>\frac{27}{4}q^2$,  the equation $x^3-px+q=0$ has three distinct real roots, say $x_1>x_2>x_3$. 
Observe that $-x_1,-x_2, -x_3$ are the roots of $x^3-px-q=0$. Thus, these two cubic polynomials have the same spread.
Without loss of generality, we can assume $q\geq 0$.
Let $\alpha=\frac{1}{3}\arccos(-\frac{q/2}{(p/3)^{3/2}})\in [\frac{\pi}{6}, \frac{\pi}{3})$. 
We have
$$\cos(3\alpha)= -\frac{q/2}{(p/3)^{3/2}}.$$
Applying the triple angle cosine formula, we have
$$4\cos^3(\alpha) - 3\cos(\alpha) = -\frac{q/2}{(p/3)^{3/2}}.$$
Plugging $\cos (\alpha)=\frac{x}{2(p/3)^{1/2}}$ and simplifying it, we get
$$x^3-px+q=0.$$
Thus $x_1=2(p/3)^{1/2} \cos (\alpha)$ is a root of Equation \eqref{eq:cubic}. Similarly,
$x_2=2(p/3)^{1/2} \cos\left(\alpha-\frac{2\pi}{3}\right)$, and $ x_3= 2(p/3)^{1/2} \cos\left(\alpha+\frac{2\pi}{3}\right)$ are also the roots of
Equation \eqref{eq:cubic}.
Since $\alpha\in [\frac{\pi}{6}, \frac{\pi}{3})$, we have
\begin{align*}
\frac{5\pi}{6}\leq \alpha+\frac{2\pi}{3} <\pi. \\
-\frac{\pi}{2}\leq \alpha-\frac{2\pi}{3} <-\frac{\pi}{3}. 
\end{align*}
Therefore
$$x_1>x_2>0>x_3.$$
In particular, we have
\begin{align*}
 S(q)&=x_1-x_3\\
 &= 2(p/3)^{1/2}\left( \cos(\alpha)-\cos (\alpha+\frac{2\pi}{3})
 \right) \\
 &= 2(p/3)^{1/2}\cdot 2 \sin\left(\frac{\pi}{3}\right)\sin\left(\alpha+\frac{\pi}{3}\right)\\
 &=2\sqrt{p} \sin\left(\alpha+\frac{\pi}{3}\right).
\end{align*}
Since $\alpha$ is an increasing function on $q$ and $S(q)$ is a decreasing function on $\alpha$, we conclude $S(q)$ is a decreasing function on $q$.
\end{proof}

We now determine all admissible pairs $(s, t)$. 

\begin{proof}[Proof of Theorem~\ref{thm:admpairs}]

We will first show the `only if' direction of Theorem~\ref{thm:admpairs}. Recall that by definition, the pair $(s,t)$ is admissible if $\psi(L) \leq 0$ for all graphs $L$ on $s-1$ vertices, and $\psi(L)=0$ only if $L=(s-1)K_1$. For $L=K_{1,s-2}$, we have that 

  \begin{align*}  \psi(K_{1, s-2}) &= 3\left((s-2)^2 + s-2\right) -\frac{2}{s-1}(2(s-2))^2 - (t-1)2(s-2)\\
  &=3(s-2)(s-1) - \frac{8}{s-1}(s-2)^2 - 2(t-1)(s-2),
  \end{align*}
from which it easily follows that $\psi(K_{1, s-2}) > 0$ if and only if \[t-1 < \frac32(s-1) - \frac{4(s-2)}{s-1}  \implies t < \frac32(s-3) + \frac{4}{s-1}.\]
Thus we can conclude that if $(s,t)$ is admissible, then $t \geq \frac{3}{2}(s-3) + \frac{4}{s-1}$.

Before we show the `if' direction of Theorem~\ref{thm:admpairs},
we need two upper bounds on the sum of the squared degrees of a graph due to de Caen~\cite{dC1998} and Das~\cite{Das2004}, respectively. 

\begin{theorem}[de Caen~\cite{dC1998}]\label{thm:decaen}
Let $G$ be a graph with $n$ vertices, $e$ edges and degrees $d_1 \ge d_2 \ge \cdots \ge d_n$. Then, 
\[\sum_{i=1}^nd_i^2 \le e\left(\frac{2e}{n-1} + n-2\right).\]
\end{theorem}

\begin{theorem}[Das~\cite{Das2004}]\label{thm:das}
Let $G$ be a graph with $n$ vertices and $e$ edges. Let $d_1$ and $d_n$ be, respectively, the highest and lowest degrees of $G$. Then, 
\[\sum_{i=1}^nd_i^2 \le 2e(d_1+d_n) - nd_1d_n.\]
 \end{theorem}

 We first prove a lemma that almost covers the entire range of $t$ using only Theorem~\ref{thm:decaen}.

\begin{lemma}\label{lem:largetadm}
If $t\geq s$ and $t \geq \frac{3}{2}(s-3) + 1$, then the pair $(s, t)$ is admissible.
\end{lemma}

\begin{proof}[Proof of Lemma~\ref{lem:largetadm}]
We may assume $s\ge 3$.
Let $L$ be a graph on $s-1$ vertices with at least one edge. By Theorem~\ref{thm:decaen},
\[3\sum_{i\in V(L)}d_i^2 \le 3\frac{\sum_{i\in V(L)}d_i}{2}\left(\frac{\sum_{i\in V(L)}d_i}{s-2} + s-3\right) = \frac{3\left(\sum_{i\in V(L)}d_i\right)^2}{2(s-2)} + \frac{3(s-3)\sum_{i\in V(L)}d_i}{2}.\]

Therefore, 
\[\psi(L) \le \left(\frac{3}{2(s-2)} - \frac{2}{s-1}\right)\left(\sum_{i\in V(L)}d_i\right)^2 + \left(\frac{3(s-3)}{2} - (t-1)\right)\sum_{i\in V(L)}d_i.\]

It follows that $\psi(L) < 0$, as $\frac{3}{2(s-2)} - \frac{2}{s-1} < 0$ for $s \geq 6$,
and by assumption $\frac{3(s-3)}{2} - (t-1) \leq 0$. 
For $s \in \{3,4,5\}$, it could be easily checked by hand that $\psi(L) <0$ for all $L$ on $s-1$ vertices with at least one edge (by computing $\psi(L)$ for all two-vertex, three-vertex and four-vertex graphs $L$). 

This implies that the pair $(s, t)$ is admissible for all $t\geq s$ and $t \geq \frac{3}{2}(s-3) + 1$. 
\end{proof}

The only cases missed by Lemma~\ref{lem:largetadm} are the following: $s\ge 10$ is even and $t = \frac32s - 4$. To take care of these cases, we use both Theorem~\ref{thm:decaen} and Theorem~\ref{thm:das}. 

Assume $t = \frac32s - 4$, where $s\ge 10$ and $s$ is even. As in the proof of Lemma~\ref{lem:largetadm}, we can use Theorem~\ref{thm:decaen} to bound $\psi(L)$ by 

\begin{equation}\label{psiineq}
\psi(L) \le \left(\frac{3}{2(s-2)} - \frac{2}{s-1}\right)\left(\sum_{i\in V(L)}d_i\right)^2 + \frac12 \sum_{i\in V(L)}d_i,
\end{equation}
where $L$ is any graph on $s-1$ vertices with at least one edge. Viewing the right-hand side of \eqref{psiineq} as a quadratic polynomial in the variable $\sum_{i\in V(L)}d_i$, we see that the quadratic polynomial has two solutions: one with $\sum_{i\in V(L)}d_i = 0$, and one with \[\sum_{i\in V(L)}d_i = \frac{\frac12}{\frac{2}{s-1}-\frac{3}{2(s-2)}} = \frac{(s-1)(s-2)}{s-5}.\]

Since $\frac{3}{2(s-2)} - \frac{2}{s-1} < 0$, it follows that if \[\sum_{i\in V(L)}d_i > \frac{(s-1)(s-2)}{s-5},\]
then $\psi(L) < 0$. Thus, assume that $\sum_{i\in V(L)}d_i \leq \frac{(s-1)(s-2)}{s-5}$, \textit{i.e.}, that the number of edges $e$ in $L$ is bounded as $e < \frac12\frac{(s-1)(s-2)}{s-5}$. 

We distinguish two cases: (1) the graph $L$ has at least two isolated vertices; (2) the graph $L$ has no isolated vertices or one isolated vertex. We assume $s \geq 12$ as the case $s=10$ can be directly checked by computer. 

Suppose $L$ has at least two isolated vertices. Let $L_{-2}$ be the graph obtained by deleting two of the isolated vertices. Then, $\sum_{i\in V(L)}d_i^2 = \sum_{i\in V(L_{-2})}d_i^2$ and $\sum_{i\in V(L)}d_i = \sum_{i\in V(L_{-2})}d_i$. We use induction on $s$. So, we assume $s\ge 12$, whence by the inductive hypothesis,
\[\psi(L_{-2}) = 3\sum_{i\in V(L_{-2})}d_i^2 - \frac{2}{(s-2)-1}\left(\sum_{i\in V(L_{-2})}d_i\right)^2 - \left(\left(\frac32(s-2)-4\right)-1\right)\sum_{i\in V(L_{-2})}d_i < 0.\]

We show $\psi(L_{-2}) > \psi(L)$. We have
\begin{align*}
    \psi(L_{-2}) - \psi(L) &= \left(\frac{2}{s-1} - \frac{2}{s-3} \right)\left(\sum_{i\in V(L)}d_i\right)^2 + 3\sum_{i\in V(L)}d_i\\
    &=\frac{-4}{(s-1)(s-3)}\left(\sum_{i\in V(L)}d_i\right)^2 + 3\sum_{i\in V(L)}d_i.
\end{align*}
Viewing $\psi(L_{-2}) - \psi(L)$ as a quadratic polynomial in $\sum_{i\in V(L)}d_i$, it follows that $\psi(L_{-2}) - \psi(L) > 0$ if $\sum_{i\in V(L)}d_i < \frac34 (s-1)(s-3)$. Indeed we have that $\sum_{i\in V(L)}d_i \leq \frac{(s-1)(s-2)}{s-5} < \frac34(s-1)(s-3)$ if $s\ge 10$. Therefore, $\psi(L) < \psi(L_{-2}) < 0$. 

Now, assume that instead $L$ has at most one isolated vertex. Recall that by our assumption, 
\[\sum_{i\in V(L)}d_i \leq \frac{(s-1)(s-2)}{s-5} = s + 2 + \frac{12}{s-5}.\]
Without loss of generality, let $d_1\geq d_2\geq \cdots \geq d_{s-1}$ be the degree sequence of $L$. We could easily check by hand that $\psi(L)<0$ for all $L$ with degree sequence of the form $(d_1, 1, \cdots, 1,1)$, $(d_1, 1, \cdots, 1,0)$, $(d_1,2, 1,1, \cdots,1, 1)$, or $(d_1, 2, 1,1, \cdots, 1,0)$. 

Otherwise, we have that $d_1 \leq s+2 + \frac{12}{s-5} - (s-1) = 3 + \frac{12}{s-5}$ if $d_{s-1} = 0$ and similarly $d_1 \leq 2 + \frac{12}{s-5}$ if $d_{s-1} = 1$. In either case, $d_1 + d_{s-1} \leq 3 + \frac{12}{s-5}$ if $s\ge 12$.
Since $d_1 + d_{s-1}$ is an integer, we have that $d_1 + d_{s-1} \leq 4$ for $s\geq 12$.
Therefore, by Theorem~\ref{thm:das}, $\sum_{i\in V(L)}d_i^2 \leq 4\sum_{i\in V(L)}d_i$, so
\begin{align*}
    \psi(L) &= 3 \sum_{i\in V(L)}d_i^2 -\frac{2}{s-1} \left(\sum_{i\in V(L)}d_i\right)^2  - \left(\frac{3}{2}s-5\right)\sum_{i\in V(L)}d_i\\
            & \leq \left(12 - \left(\frac{3}{2}s-5\right)\right)\sum_{i\in V(L)}d_i -\frac{2}{s-1}\left(\sum_{i\in V(L)}d_i\right)^2.
\end{align*}
But now we see that $\psi(L) < 0$, since $-\frac{2}{s-1}\left(\sum_{i\in V(L)}d_i\right)^2 < 0$ and $17 - \frac32 s < 0$ if $s\ge 12$. This completes the proof of Theorem~\ref{thm:admpairs}.
 
\end{proof}

\section*{Acknowledgment}
We thank the anonymous referees for their careful reading of the paper and their helpful comments.

\end{document}